\documentclass{birkmult}

\usepackage{amssymb,amsfonts, amsmath, amsthm, wasysym, mathtools}

\def\Z{\mathbb{Z}}
\def\N{\mathbb{N}}
\def\R{\mathbb{R}}

\def\L{\mathcal{L}}

\newcommand{\comment}[1]{}
\newcommand{\du}[2]{\left\langle #1, #2 \right\rangle}

\DeclareMathOperator{\Sp}{Sp}

\DeclareMathOperator{\Lex}{Lex}
\DeclareMathOperator{\supp}{supp}

\newtheorem{theorem}{Theorem}[section]
\newtheorem{lemma}[theorem]{Lemma}
\newtheorem{corol}[theorem]{Corollary}
\newtheorem{prop}[theorem]{Proposition}

\theoremstyle{definition}

\newtheorem{defn}[theorem]{Definition}

\newtheorem{remark}[theorem]{Remark}

\begin{document}

\title[Lexicographic cones and the ordered projective tensor product]{Lexicographic cones and the \\ ordered projective tensor product}

\author{Marten Wortel}

\address{Department of Mathematics and Applied Mathematics \\ University of Pretoria \\
Private Bag X20 Hatfield \\
0028 Pretoria \\
South Africa}

\email{marten.wortel@up.ac.za}

\dedicatory{Dedicated to Ben de Pagter on the occasion of his 65th birthday}

\begin{abstract}
We introduce lexicographic cones, a method of assigning an ordered vector space $\Lex(S)$ to a poset $S$, generalising the standard lexicographic cone. These lexicographic cones are then used to prove that the projective tensor cone of two arbitrary cones is a cone, and to find a new characterisation of finite-dimensional vector lattices. 
\end{abstract}

\subjclass{Primary 06F20; Secondary 46A40, 46M05}

\keywords{Lexicographic cone, finite-dimensional vector lattices, ordered projective tensor product}

\maketitle

\section{Introduction}

In the theory of Archimedean vector lattices, the Fremlin projective tensor product is an important tool with many applications. The Fremlin projective tensor product was introduced by Fremlin in \cite{fremlin}, and it satisfies the usual universal property for Riesz bimorphisms. However, the construction of the Fremlin tensor product is fairly complicated and uses representation theory.

In \cite{grobler_labuschagne}, Grobler and Labuschagne gave an easier construction of the Fremlin projective tensor product. A crucial ingredient in their construction is the wedge generated by tensors of positive elements in the algebraic tensor product, called the projective cone (the projective cone of ordered vector spaces was introduced and investigated earlier, cf.\ \cite{nakano, schaefer2, merklen, ellis, peressini_sherbert,birnbaum}). Amongst other things, they show that the projective cone of two Archimedean ordered vector spaces with the Riesz decomposition property is actually a cone (\cite[Theorem~2.5]{grobler_labuschagne}); despite its name, it is a priori not clear at all that the projective cone is a cone. In \cite[Theorem~3.3]{onno_anke}, this result was extended by van Gaans and Kalauch to Archimedean ordered vector spaces, removing the Riesz decomposition requirement.

In this paper we go one step further and show that the projective tensor cone of two arbitrary ordered vector spaces is a cone, removing the Archimedean requirement. Our main tool is lexicographic cones, which is a method for assigning an ordered vector space $\Lex(S)$ to any poset $S$. In finite dimensions, by choosing $S$ appropriately, this generates the standard lexicographic cone $\R^d_{lex}$, the standard cone $\R^d_+$, and many new intermediate ordered vector spaces. It turns out that the projective tensor product of these lexicographic cones has a very nice description, cf.\ Proposition~\ref{p:product_lexico}, which allows us to prove the above mentioned result.

In \cite[Theorem~3.9]{schaefer}, Schaefer gave a recursive characterisation of finite-dimensional vector lattices. It turns out that these can be reformulated in terms of $\Lex(S)$ for appropriate $S$, so the lexicographic cones also yield an alternative, direct characterisation of finite-dimensional vector lattices, cf.\ Theorem~\ref{thm:fin_dim_lex}.

We briefly explain the structure of the paper. In Section~\ref{s:lex_cones}, we start by introducing $\Lex(S)$ and proving some basic properties of these ordered vector spaces; we also investigate the dual cone of these spaces. We characterise when $\Lex(S)$ is a vector lattice in Section~\ref{sec:fin_dim_lattices} and prove the characterisation of finite-dimensional vector lattices mentioned above. In Section~\ref{s:proj_tensor_prod}, we investigate the projective tensor product of these lexicographic cones and prove the main result that this is a cone.

\section{Lexicographic cones}\label{s:lex_cones}

A \emph{wedge} $C$ in a vector space $X$ is a convex subset satisfying $C + C \subset C$ and $\lambda C \subset C$ for all $\lambda \geq 0$. A wedge $C$ is called a \emph{cone} if $C \cap -C = \{0\}$. An \emph{ordered vector space} is a vector space $X$ equipped with a linear order, i.e., if $x,y \in X$ and $x \leq y$, then $x+z \leq y+z$ for all $z \in X$ and $\lambda x \leq \lambda y$ for all $\lambda \geq 0$. A linear order on $X$ generates the cone of positive elements $C \coloneqq \{x \in X \colon x \geq 0\}$, and conversely, every cone $C$ generates a linear order defined by $x \leq y$ if and only if $y-x \in C$.

In this paper, $S$ will always denote a poset (partially ordered set). For $s \in S$, we denote the set $\{t \in S \colon t < s\}$ by $\langle s)$ and the set $\{t \in S \colon t \leq s\}$ is denoted by $\langle s]$. The symbols $(s \rangle$ and $[s \rangle$ have similar meaning.

Let $F_0(S)$ be the vector space of finitely supported real-valued functions on $S$. Then $\Lex(S)$ is defined to be the vector space $F_0(S)$ equipped with the cone
$$
\Lex(S)_+ \coloneqq \{f \in F_0(S) \colon f(s) < 0 \Rightarrow  \exists t<s \text{ with } f(t) >0 \}. 
$$

\begin{lemma}\label{l:lex_is_cone}
$\Lex(S)_+$ is a cone.
\end{lemma}
\begin{proof}
To show that $\Lex(S)_+$ is a wedge, let $f, g \in \Lex(S)_+$. If $s_0 \in I$ is such that $(f+g)(s_0) < 0$, then either $f(s_0) < 0$ or $g(s_0) < 0$; assume the first case. Then there is an $s_1 < s_0$ with $f(s_1) > 0$. If $g(s_1) \geq 0$ then $(f+g)(s_1) > 0$, and if $g(s_1) < 0$ then there is an $s_2 < s_1$ with $g(s_2) > 0$. If $f(s_2) \geq 0$ then $(f+g)(s_2) > 0$, and if $f(s_2) < 0$ then there is an $s_3 < s_2$ with $f(s_3) > 0$. Hence either $(f+g)(s_k) > 0$ for some $k \in \N$, or $f(s_{2n+1}) > 0$ and $g(s_n) > 0$ for all $n \in \N$, the latter case contradicting the fact that $f$ and $g$ are finitely supported. Since $s_k < s_0$, $f+g \in \Lex(S)_+$.

We now show that $\Lex(S)_+$ is a cone, so suppose $\pm f \in \Lex(S)_+$ and that $f(s_0) < 0$ for some $s_0 \in S$. Then $f(s_1) > 0$ for some $s_1 < s_0$, and so $-f(s_1) < 0$ which implies that $-f(s_2) > 0$ for some $s_2 < s_1$. Repeating this argument shows that $f$ is supported on an infinite set, which contradicts $f \in F_0(S)$. Hence $\Lex(S)_+$ is a cone.
\end{proof}

Denote by $e_s$ the function $t \mapsto \delta_{st}$, then $\{e_s\}_{s \in S}$ forms a basis for $\Lex(S)$. If $S = \{1, \ldots, d\}$ with the standard ordering, then $\Lex(S)$ is the usual lexicographic cone $\R^d_{lex}$, whereas if $S = \{1, \ldots, d \}$ with no elements comparable, then $\Lex(S)$ is the standard cone $\R^d_+$. Generalising the previous example, if $S$ is an arbitrary disjoint union of posets $S_k$, then it is easy to see that $\Lex(S) \cong \bigoplus_k \Lex(S_k)$ (this direct sum is an order direct sum: an element is positive if and only if all its components are positive).

Let $F(S)$ be the vector space of real-valued functions on $S$. Under the natural duality $\du{f}{g} \coloneqq \sum_{s \in S} f(s)g(s)$, the space $F(S)$ can be identified with the algebraic dual of $F_0(S)$. We now identify $\Lex(S)_+^*$, the dual cone of $\Lex(S)_+$. The set $F(S)_+$ denotes the functions $g \in F(S)$ for which $g(s) \geq 0$ for all $s \in S$.

\begin{lemma}\label{l:dual_cone}
$\Lex(S)_+^* \cong \{g \in F(S)_+ \colon \supp(g) \subset \{s \in S \colon s \text{ is minimal} \} \}$
\end{lemma}

\begin{proof}
Let $g$ be supported on the minimal elements of $S$. Any $f \in \Lex(S)_+$ is nonnegative on minimal elements of $S$, so $\du{f}{g} \geq 0$ for all $f \in \Lex(S)_+$, hence $g \in \Lex(S)_+^*$. Conversely, if $g(s) > 0$ for a nonminimal element $s$ and $t < s$, then $f_n \coloneqq e_t - n e_s \in \Lex(S)_+$ and $\du{f_n}{g} < 0$ for large enough $n$, so $g \notin \Lex(S)_+^*$.
\end{proof}

In particular, if $S$ contains no minimal elements, then $\Lex(S)_+^*$ is trivial.

\section{Lattices}\label{sec:fin_dim_lattices}

First we will investigate when $\Lex(S)$ is a vector lattice. Let $\bigwedge := \{s,t,m\}$ where $s,t < m$ are the only nontrivial relations. Then $\Lex(\bigwedge) \cong (\R^3, C)$ where 
$$ C := \{ (x,y,z) \colon x,y \geq 0, z \in \R\} \setminus \{(0,0,z) \colon z < 0 \}. $$
The set of upper bounds of the zero vector and $(1,-1,-1)$ equals 
$$ \{(x,y,z) \colon x \geq 1, y \geq 0, z \in \R \}, $$
and if $(x,y,z)$ is in this set, then $(x,y,z-1)$ is a smaller element in this set, so it has no least element. Therefore the zero vector and $(1,-1,-1)$ have no supremum, and so $\Lex(\bigwedge)$ is not a vector lattice. We will show that this is in some sense the only possible counterexample.

\begin{defn}
A poset $S$ is called a \emph{forest} if for each $s \in S$, the set $\langle s)$ is totally ordered. A forest is called a \emph{tree} if every two elements have a common lower bound. A \emph{root} of a tree $S$ is a minimal element of $S$.
\end{defn}

The forests are precisely the disjoint unions of trees. Note that a tree may not have a root (e.g.\ $\Z$), but if it exists it is unique, and every finite tree has a root. 

\begin{remark}
Note that a common definition of a tree in the literature requires the initial segments to be well-ordered, not just totally ordered. For finite sets, both definitions coincide.
\end{remark}

If $S$ is a forest, then for $m \in S$, the set $\langle m)$ is totally ordered, hence $S$ contains no subposet isomorphic to $\bigwedge$. If $S$ is not a forest, then for some $m \in S$, two predecessors are incomparable, hence $S$ contains a subposet isomorphic to $\bigwedge$. Therefore $S$ is a forest if and only if it contains no subposet isomorphic to $\bigwedge$.

\begin{remark}\label{rem:sup_on_supp}
In the next theorem, we will be concerned with possible suprema of a function $f \in \Lex(S)$ and $0$. We claim that any upper bound $g$ not supported on $\supp(f)$ is not a supremum. Indeed, if $s \notin \supp(f)$ and $g(s) > 0$, then $g - (g(s)/2) e_s$ is a lower upper bound of $f$ and $0$, and if $g(s) < 0$, then $g - e_s$ is a lower upper bound of $f$ and $0$. Hence it suffices to only consider functions supported on $\supp(f)$, and so we may assume that $S = \supp(f)$.
\end{remark}

\begin{theorem}
$\Lex(S)$ is a vector lattice if and only if $S$ is a forest.
\end{theorem}
\begin{proof}
Suppose $S$ is not a forest. Let $\{s,t,m\} \subset S$ be isomorphic to $\bigwedge$, and let $f := e_s - e_t - e_m$. By Remark~\ref{rem:sup_on_supp} and the example at the beginning of this section, $f$ and $0$ have no supremum.

Conversely, suppose that $S$ is a forest, and let $f \in \Lex(S)$; we will compute $f \vee 0$. By Remark~\ref{rem:sup_on_supp} we may assume that $S = \supp(f)$ is a finite forest, which is a disjoint union of finite trees $S_k$. Then $\Lex(S)$ is a finite order direct sum of $\Lex(S_k)$, and since the order is coordinatewise, it suffices to consider a single $\Lex(S_k)$. Let $s \in \supp(f)$ be the root of $S_k$. If $f(s) > 0$, then $f > 0$, and so $f \vee 0 = f$, and if $f(s) < 0$, then $f < 0$, and so $f \vee 0 = 0$.
\end{proof}

The class $\Lex(S)$ where $S$ is a forest without minimal element is a class of vector lattices with no nontrivial positive functionals (cf.\ Lemma~\ref{l:dual_cone}); a particular example is $\Lex(\Z)$. (Another example is $L^p(0,1)$ for $0<p<1$; here one can show that positive functionals are automatically continuous and that there are no nontrivial continuous functionals.)

Each finite forest $S$ yields a finite-dimensional vector lattice $\Lex(S)$, and we will show that those are the only finite-dimensional vector lattices. If $X$ is a finite-dimensional vector lattice, then \cite[Theorem~3.9]{schaefer} shows that $X$ is a direct sum of $\R \circ M$'s, where $\circ$ denotes the lexicographic union and $M$ is a maximal ideal in $\R \circ M$. Each maximal ideal is of the same form, so this yields a recursive characterisation of finite-dimensional vector lattices. Our result below yields an alternative, non-recursive characterisation.

\begin{theorem}\label{thm:fin_dim_lex}
A ordered vector space $X$ is a finite-dimensional vector lattice if and only if it is isomorphic to $\Lex(S)$ for some finite forest $S$.
\end{theorem}

\begin{proof}
If $S$ is a finite forest, then the above theorem shows that $\Lex(S)$ is a vector lattice. (Alternatively, one can prove this directly: let $s_1, \ldots, s_n$ be the roots of $S$. Let $M_k := \Lex(\{ s \in S \colon s > s_k\})$, then $S \cong \oplus_{k=1}^n \, \R \circ M_k$. Now the ordered vector space $M_k$ is of smaller dimension, so it follows by induction on the dimension.)

Conversely, suppose $X$ is a finite-dimensional vector lattice. By induction on $\dim(X)$ we will show that it is isomorphic to $\Lex(S)$ for some finite forest $S$. The result is obvious if $\dim(X)=1$. Suppose it holds for all dimensions $1 \leq k \leq d$ and suppose $\dim(X) = d+1$, then $X \cong \oplus_{k=1}^n \, \R \circ M_k$ by \cite[Theorem~3.9]{schaefer}. The induction hypothesis now implies that $M_k \cong \Lex(S_k)$ for some finite forest $S_k$. Let $S'_k$ be the forest $S_k$ adjoined with a new element that is below every element in $S_k$, then $\R \circ M_k \cong \Lex(S'_k)$ by the definition of lexicographic union. Hence $X \cong  \oplus_{k=1}^n \Lex(S'_k) \cong \Lex(S)$ where $S$ is the forest defined by the disjoint union of the trees $S'_k$.
\end{proof}

\section{Projective tensor product}\label{s:proj_tensor_prod}

If $X$ and $Y$ are ordered vector spaces, then the \emph{projective tensor product} of $X$ and $Y$ is the vector space $X \otimes Y$ equipped with the \emph{projective cone} 
$$
K_p = K_p(X,Y) \coloneqq \left\{ \sum_{i=1}^n x_i \otimes y_i \colon n \in \N, x_i \in X_+, \in Y_+ \right\}.
$$
It is obvious that $K_p$ is a wedge.

Our next goal is to show that $K_p$ is actually a cone. Let $X$ be an ordered vector space. A set $G \subset X_+$ is a \emph{generating set} for $X_+$ if every $x \in X_+$ can be written as a positive linear combination of elements of $G$. A set of positive elements such that its positive linear span contains a generating set is obviously generating. Suppose $G$ is a generating set for $X_+$ and $H$ is a generating set for $Y_+$. Then if $u = \sum_{i=1}^n x_i \otimes y_i \in K_p \subset X \otimes Y$, then there exist $\lambda_{ij} \geq 0$, $g_{ij} \in G$ and $h_{ij} \in H$ such that $x_i = \sum_{j=1}^{n_i} \lambda_{ij} g_j$ and $y_i = \sum_{k=1}^{m_i} \mu_{ik} h_k$. Therefore
$$ u = \sum_{i=1}^n x_i \otimes y_i = \sum_{i=1}^n \sum_{j=1}^{n_i} \sum_{k=1}^{m_i} \lambda_{ij} \mu_{ik} (g_{ij} \otimes h_{ik}), $$
and so $G \otimes H := \{g \otimes h \colon g \in G, h \in H\}$ is a generating set for $K_p(X,Y)$.

\begin{lemma}\label{l:generating_set_lex}
The set
$$ G := \{ e_s - \lambda e_t \colon \lambda > 0, s < t \} \cup \{e_s \colon s \in S \} $$
is a generating set for $\Lex(S)_+$.
\end{lemma}
\begin{proof}
We will show that every $f \in \Lex(S)_+$ can be written as a positive linear combination of elements of $G$ by induction on $|\supp(f)|$. If $|\supp(f)| = 1$ then the result is obvious. Suppose it holds for every $f$ with $ 1 \leq |\supp(f)| \leq n$, and take an $f \in \Lex(S)_+$ with $|\supp(f)| = n+1$. Then there is an $s \in S$ with $f(s) > 0$. Write $f = f|_{[s\rangle} + f|_{[s\rangle^c}$. Clearly $f|_{[s\rangle}$ is positive since $s$ is the smallest element in its support and $f(s) > 0$. To show that $f|_{[s\rangle^c}$ is positive, let $t \in [s\rangle^c$ with $f(t) < 0$, then there is a $r < t$ with $f(r) >0$. If $r > s$, then $t > r > s$, contradicting $t \in [s\rangle^c$, so $r \in [s\rangle^c$, and therefore $f|_{[s\rangle^c} \geq 0$.

Let $T^- := \{ t \in S \colon t > s, f(t) < 0\}$ and $T^+ := \{ t \in S \colon t > s, f(t) > 0\}$. If $T^- = \emptyset$ then 
$$ f|_{[s\rangle} = \sum_{t \in T^+} f(t) e_t + f(s)e_s, $$
and if $T^- \not= \emptyset$ then 
$$ f|_{[s\rangle} = \sum_{t \in T^+} f(t) e_t + \sum_{t \in T^-} \left( \frac{f(s)}{|T^-|} e_s - |f(t)| e_t \right) . $$
In both cases $f|_{[s\rangle}$ can be written as a positive linear combination of elements of $G$, and the same holds for $f|_{[s\rangle^c}$ by the induction hypothesis. Hence $f = f|_{[s\rangle} + f|_{[s\rangle^c}$ can be written as a positive linear combination of elements of $G$.

\comment{Let $0 \not= f \in Lex(S)_+$, then there is an $s \in S$ with $f(s) > 0$. Write $f = f|_{[s\rangle} + f|_{[s\rangle^c}$. Clearly $f|_{[s\rangle}$ is positive since $s$ is the smallest element in its support and $f(s) > 0$. To show that $f|_{[s\rangle^c}$ is positive, let $t \in [s\rangle^c$ with $f(t) < 0$, then there is a $r < t$ with $f(r) >0$. If $r > s$, then $t > r > s$, contradicting $t \in [s\rangle^c$, so $r \in [s\rangle^c$, and therefore $f|_{[s\rangle^c} \geq 0$. It is clear that $f|_{[s\rangle}$ can be written as a positive linear combination of elements of $G$. Moreover, $f|_{[s\rangle^c}$ can be decomposed like $f$ and has smaller support than $f$, so the result follows by induction on the size of the support of $f$.}
\end{proof}

Using this lemma we can show that the projective tensor product behaves well with respect to lexicographic cones.

\begin{prop}\label{p:product_lexico}
Let $S$ and $T$ be posets. Then there is a natural linear isomorphism $\Lex(S) \otimes \Lex(T) \cong \Lex(S \times T)$, and under this isomorphism, $K_p \subset \Lex(S) \otimes \Lex(T)$ satisfies $K_p \cong \Lex(S \times T)_+$.
\end{prop}
\begin{proof}
Let $\{e_s\}_{s \in I}$, $\{f_t\}_{t \in T}$ and $\{g_{s,t}\}_{(s,t) \in S \times T}$ be the natural bases of $\Lex(S)$, $\Lex(T)$ and $\Lex(S \times T)$; then $e_s \otimes f_t \mapsto g_{s,t}$ induces the natural linear isomorphism $\Lex(S) \otimes \Lex(T) \cong \Lex(S \times T)$.

By Lemma~\ref{l:generating_set_lex},
$$G = \{e_{s_1} - \lambda e_{s_2}: \lambda > 0, s_1 < s_2 \} \cup \{e_s: s \in S\}$$
  is a generating set for $\Lex(S)$, 
  $$H = \{f_{t_1} - \lambda f_{t_2}: \lambda > 0, t_1 < t_2 \} \cup \{f_t: t \in T\}$$
  is a generating set for $\Lex(T)$, and a generating set of $\Lex(S \times T)$ is given by
  $$ \{g_{s_1,t_1} - \alpha g_{s_2,t_2}: \alpha > 0, (s_1, t_1) < (s_2, t_2) \} \cup \{g_{s,t} : (s,t) \in S \times T \}. $$
     We will compute $G \otimes H$ in $\Lex(S \times T)$. Let $s \in S$, $t \in T$, $s_1 < s_2$, $t_1 < t_2$ and $\lambda, \mu > 0$, then
\begin{align*}
(e_{s_1} - \lambda e_{s_2}) \otimes (f_{t_1} - \mu f_{t_2}) &\cong g_{s_1,t_1} - \lambda g_{s_2,t_1} - \mu g_{s_1,t_2} + \lambda \mu g_{s_2,t_2} \\
(e_{s_2} - \lambda e_{s_2}) \otimes f_t &\cong g_{s_1,t} - \lambda g_{s_2,t} \\
e_s \otimes (f_{t_1} - \mu f_{t_2}) &\cong g_{s,t_1} - \mu g_{s,t_2} \\
e_s \otimes f_t &\cong g_{s,t}.
\end{align*}
All elements on the right-hand side are positive in $\Lex(S \times T)$. It now suffices to show that the positive span of these elements contains a generating set for $\Lex(S \times T)$. For that, we have to show that for $s_1 < s_2$ and $t_1 < t_2$ (the case where $s_1 = s_2$ or $t_1 = t_2$ is already covered by the second and third line above) and $\alpha > 0$, $g_{s_1,t_1} - \alpha g_{s_2,t_2}$ is a positive linear combination of these elements. Indeed,
$$ g_{s_1,t_1} - \alpha g_{s_2,t_2} = (g_{s_1,t_1} - g_{s_2,t_1}) + (g_{s_2,t_1} - \alpha g_{s_2,t_2} ). \eqno\qedhere $$
\end{proof}

For the next theorem we need a proposition about maximal cones in finite-dimensional vector spaces, and for this we need some preparations.

\begin{lemma}\label{lem:max_cones}
Let $C$ be a cone in a vector space $X$. Then $C$ is contained in a maximal cone. Moreover, the following are equivalent:
\item[(i)] $C$ is maximal;
\item[(ii)] $C \cup -C = X$;
\item[(iii)] $(X,C)$ is totally ordered.
\end{lemma}

\begin{proof}
The first statement easily follows from Zorn's Lemma. 

To prove $(i) \Rightarrow (ii)$, suppose $x \in X \backslash (C \cup -C)$. Towards showing that the wedge generated by $C$ and $x$ is a cone, let $\lambda, \mu \geq 0$ and $y,z \in C$ be such that $\lambda x + y = -\mu x - z$. Then $(\lambda+\mu)x = -y-z \in -C$, which is only possible if $\lambda = \mu = 0$ and $y = z = 0$. Hence the wedge generated by $x$ and $C$ is a cone, contradicting the maximality of $C$.

$(ii) \Rightarrow (i)$ is obvious.

To prove $(ii) \Leftrightarrow (iii)$, note that $(X,C)$ being totally ordered means that every element of $X$ is either positive or negative, i.e., $C \cup -C = X$.
\end{proof}

\begin{prop}\label{p:lex_maximal}
Let $X$ be an ordered vector space of dimension $d$. Then $X_+$ is contained in a maximal cone which is isomorphic to $\R^d_{lex}$.
\end{prop}

\begin{proof}
By Lemma~\ref{lem:max_cones}, $X_+$ is contained in a maximal cone $C$ which induces a total order. A totally ordered vector space is obviously a vector lattice, so $C \cong \Lex(S)_+$ by Theorem~\ref{thm:fin_dim_lex}. The total order forces $S$ to have no incomparable elements, hence $S$ must be equal to $\{1,\ldots, d\}$ with the usual order, and so $\Lex(S)_+ \cong \R^d_{lex}$.
\end{proof}

An alternative proof of Proposition~\ref{p:lex_maximal}, which does not rely on the results from Section~\ref{sec:fin_dim_lattices}, follows from the following results, whose easy proof is left to the reader.

\begin{lemma}
If $W$ is a wedge that is not dense in a locally convex space $X$, then there exists an $x^* \in X^*$ such that the closed half-space $\{x \in X \colon \du{x}{x^*} \geq 0\}$ contains $W$.
\end{lemma}

\begin{lemma}
If $W$ is a wedge contained in a finite-dimensional space $X$, then $W$ is dense in $X$ if and only if $W = X$.
\end{lemma}

\begin{corol}
If $C$ is a cone in a finite-dimensional vector space $X$, then it is contained in a closed half-space.
\end{corol}

Note that it is in general not true that a cone is contained in a half-space, since the half-space would induce a positive functional and some cones (even lattice cones) do not admit any positive functionals, as shown before.

The alternative proof of Proposition~\ref{p:lex_maximal} now follows easily by induction on the dimension and the above corollary.

\begin{theorem}\label{t:tensor_product_cone}
Let $X$ and $Y$ be ordered vector spaces. Then $K_p(X,Y)$ is a cone.
\end{theorem}
\begin{proof}
Let $\pm u \in K_p(X,Y)$. Then $u = \sum_{i=1}^k x_i \otimes y_i$ and $u = - \sum_{i=k+1}^n x_i \otimes y_i$, for $x_i \in X_+$ and $y_i \in Y_+$. Hence $\pm u \in K_p(E,F)$, where $E = \Sp\{x_i\}$ and $F = \Sp\{y_i\}$ are finite dimensional. Therefore, to show the theorem, we may assume that $X$ and $Y$ are finite dimensional.

In this case, by Proposition~\ref{p:lex_maximal}, $X_+$ and $Y_+$ are contained in lexicographic cones, of which the tensor product is of the form $\Lex(\{1, \ldots, d_X\} \times \{1, \ldots, d_Y \})_+$ by Proposition~\ref{p:product_lexico}. This is a cone by Lemma~\ref{l:lex_is_cone}, and so $K_p$, as a subwedge of this cone, is a cone.
\end{proof}

\bibliographystyle{alpha}
\bibliography{tensorovsbib}

\end{document}